\documentclass[11pt,twoside,a4paper]{amsart}
\usepackage[initials,nobysame]{amsrefs}
\usepackage{amssymb}
\usepackage[utf8]{inputenc}
\date{\today}

\usepackage[all,cmtip]{xy}
\usepackage{graphics}

\DeclareMathOperator{\Id}{Id}

\title[Residue currents and cycles of complexes of vector bundles]
{Residue currents and cycles of complexes \\ of vector bundles}

\author{Richard L\"ark\"ang \& Elizabeth Wulcan}
\thanks{The authors were partially supported by the Swedish Research Council.}
\address{Department of Mathematical Sciences, University of Gothenburg and
Chalmers University of Technology, SE-412 96 G\"{o}teborg, Sweden}
\email{larkang@chalmers.se, wulcan@chalmers.se}

\usepackage{geometry}
\newtheorem{thm}{Theorem}[section]
\newtheorem{lma}[thm]{Lemma}
\newtheorem{cor}[thm]{Corollary}
\newtheorem{prop}[thm]{Proposition}

\theoremstyle{definition}

\theoremstyle{remark}

\newtheorem{preremark}[thm]{Remark}
\newtheorem{preex}[thm]{Example}

\newenvironment{ex}{\begin{preex}}{\qed\end{preex}}

\newcommand{\F}{\mathcal{F}}

\newcommand{\C}{\mathbb{C}}

\newcommand{\R}{\mathbb{R}}
\newcommand{\J}{\mathcal{J}}
\newcommand{\E}{\mathcal{E}}

\newcommand{\Ok}{\mathcal{O}}
\newcommand{\w}{{\wedge}}
\newcommand{\codim}{{\text{codim}\,}}
\newcommand{\Hom}{{\text{Hom}\,}}
\newcommand{\End}{{\text{End}\,}}

\usepackage{xcolor}

\def\dbar{\bar{\partial}}

\DeclareMathOperator{\supp}{supp}
\DeclareMathOperator{\tr}{tr}
\DeclareMathOperator{\length}{length}
\DeclareMathOperator{\im}{im}
\DeclareMathOperator{\coker}{coker}

\numberwithin{equation}{section}

\begin{document}
\nocite{*}
\bibliographystyle{plain}

\begin{abstract}
    We give a factorization of the cycle of a bounded complex of
    vector bundles in terms of certain associated differential forms
    and residue currents.
This is a generalization of previous results
    in the case when the complex is a locally free resolution of the
    structure sheaf of an analytic space and it can be seen as
    a generalization of the classical Poincar\'e-Lelong formula.
\end{abstract}

\maketitle

\section{Introduction}

Given a holomorphic function $f$ on a complex manifold $X$, recall
that the classical \emph{Poincar\'e-Lelong formula} asserts that
$\dbar \partial \log |f|^2 = 2\pi i[Z]$,
where $\Ok_Z = \Ok_X/\J(f)$, $\J(f)$ is the ideal generated by $f$,
and $[Z]$ is the current of integration along $Z$,
counted with multiplicities or, more precisely, the (fundamental) cycle of $Z$.
Formally we can rewrite the Poincar\'e-Lelong formula as
\begin{equation}\label{poin}
\frac{1}{2\pi i} \dbar\frac{1}{f} \wedge df = [Z].
\end{equation}
This factorization of $[Z]$ can be made rigorous if we construe $\dbar(1/f)$
as the \emph{residue current} of $1/f$, where $1/f$ is the principal
value distribution
as introduced by Dolbeault, \cite{Dol}, and Herrera and Lieberman, \cite{HL}.
The current $\dbar(1/f)$ satisfies that a holomorphic
function $g$ on $X$ is $0$ in $\Ok_Z$ if and only if $g\dbar
(1/f)=0$. This is referred to as the \emph{duality principle} and it is
central to many applications of residue currents; in a way
$\dbar(1/f)$ can be thought of as a current representation of the
structure sheaf
~$\Ok_Z$.

In this article we give a similar analytic formula for the
cycle of any bounded complex of vector bundles.
The \emph{cycle} of the coherent sheaf $\mathcal{F}$ on $X$
is the cycle
\begin{equation*}
    [\mathcal{F}] = \sum_i m_i [Z_i],
\end{equation*}
where $Z_i$ are the irreducible components of $\supp
\mathcal{F}$, and $m_i$ is the \emph{geometric multiplicity} of $Z_i$ in
$\mathcal{F}$.
For generic $z \in Z_i$, $\mathcal{F}$ can locally be given the
structure of a free $\Ok_{Z_i}$-module of
constant rank, and $m_i$ is 
this rank.
Alternatively, expressed in an algebraic manner,
$m_i=\length_{\Ok_{Z,Z_i}}( \mathcal{F}_{Z_i})$,
see, e.g., \cite{K}*{Section~2}.
If $\mathcal F=\Ok_Z$, then $[\mathcal F]$ coincides with the cycle of $Z$, cf., e.g., \cite{Fulton}*{Chapter~1.5}.

Next, let
\begin{equation}\label{komplexet}
    0 \to E_{N} \stackrel{\varphi_{N}}{\longrightarrow} E_{N-1} \to \cdots \to
    E_1 \stackrel{\varphi_1}{\longrightarrow} E_{0}\to 0
\end{equation}
be a generically exact complex
of vector bundles on $X$. We let the
\emph{cycle} of $(E,\varphi)$ be the cycle
\begin{equation} \label{cycleofcomplex}
    [E] := \sum (-1)^\ell [\mathcal{H}_\ell(E)],
\end{equation}
where $\mathcal{H}_\ell(E)$ is the homology group of $(E,\varphi)$ at level $\ell$.
We have not found the definition of such a cycle in this setting in the literature. 
However, when all the homology groups have support at a single point, our cycle simply 
corresponds to the Euler characteristic of the complex, and \eqref{cycleofcomplex}
appears to be a natural generalization for general complexes.
Note that if $(E,\varphi)$ is a locally free resolution of a
coherent sheaf $\mathcal{F}$, i.e., it is exact at all levels $>0$ and
$\mathcal H_0(E) \cong \F$,
then $[E] = [\mathcal{F}]$.

If the $E_\ell$
are equipped with hermitian metrics, we say that $(E,\varphi)$ is a
\emph{hermitian} complex.
Given a hermitian complex $(E,\varphi)$ that is exact
outside a subvariety $Z\subset X$, in \cite{AW1} Andersson and
the second author introduced an associated residue current $R=R^E$
with support on $Z$,
that takes values in $\End E$, where $E=\oplus E_k$, and that in some
sense measures the exactness of $(E, \varphi)$.
In particular, if $(E,\varphi)$ is a locally free resolution of $\F$,
then the component $R^\ell_k$ 
that takes values in
$\Hom (E_\ell, E_k)$ vanishes if
$\ell >0$ and $R$ satisfies a duality principle for $\mathcal F$.

If $f$ is a holomorphic function on $X$
and $E_0\cong\Ok_X$ and $E_1\cong \Ok_X$ are trivial line bundles, then
\begin{equation*}
0 \xrightarrow[]{} \Ok_X \xrightarrow[]{\varphi_1} \Ok_X\to 0,
\end{equation*}
where $\varphi_1$ is the $1\times 1$-matrix $[f]$, gives a locally free resolution of $\mathcal
O_Z=\mathcal O/\mathcal J(f)$.
In this case (the coefficient of) $R=R^0_1$ is just $\dbar (1/f)$, and
the Poincar\'e-Lelong formula
\eqref{poin} can be written as\footnote{For an explanation of the relation between the signs
    in \eqref{poin} and \eqref{snodd}, see \cite{LW}*{Section 2.5}, cf.\
  Section \ref{ssect:super}.}
\begin{equation}\label{snodd}
\frac{1}{2\pi i} d\varphi_1 R^0_1=[E].
\end{equation}
Our main result is the following generalization of \eqref{snodd}.
Recall that a coherent sheaf $\F$ has pure dimension $d$ if $\supp
\mathcal \F$ has pure dimension $d$. Given an $\End E_\ell$-valued
current $\alpha$ let $\tr \alpha$ denotes the trace of
$\alpha$.

\begin{thm} \label{thm:main}
    Let $(E,\varphi)$ be a hermitian complex of vector bundles
    \eqref{komplexet} such that all its homology groups $\mathcal{H}_\ell(E)$
    have pure codimension $p > 0$ or vanish,
    and let $D$ be the connection on $\End E$ induced by arbitrary $(1,0)$-connections\footnote{See \eqref{DEnd} for how this connection is defined. } on $E_0,\ldots,E_N$.
    Then
    \begin{equation} \label{eq:main}
        \frac{1}{(2\pi i)^p p!} \sum_{\ell=0}^{N-p}
 (-1)^\ell \tr D\varphi_{\ell+1} \cdots D\varphi_{\ell+p} R^{\ell}_{\ell+p} = [E].
    \end{equation}
\end{thm}

Note that the endomorphisms $D\varphi_{\ell+1} \cdots
D\varphi_{\ell+p}$ depend on the choice of connections on $E_0,\ldots,E_N$ and the currents $R^\ell_{\ell+p}$ in general depend on the choice of hermitian
metrics on $E_0,\ldots,E_N$. There is no assumption of any relation
between the connections and the hermitian
metrics.

The proof of Theorem~\ref{thm:main}, which occupies Section
\ref{stiga},
is by induction over the number of nonvanishing homology groups
$\mathcal{H}_\ell(E)$.
The basic case is the special case when $(E,\varphi)$ has
nonvanishing homology only at level $0$.

\begin{thm} \label{thm:higherrank}
    Let $\mathcal{F}$ be a coherent sheaf of pure codimension $p$, let $(E,\varphi)$ be a hermitian
locally free resolution of $\mathcal F$,
    and let $D$ be the connection on $\End E$ induced by arbitrary $(1,0)$-connections on $E_0,\ldots,E_N$.
    Then
    \begin{equation}\label{nordbotten}
        \frac{1}{(2\pi i)^p p!} \tr D\varphi_1 \cdots D\varphi_p R^0_p = [\mathcal{F}].
    \end{equation}
\end{thm}

In \cite{LW} we gave a proof of Theorem~\ref{thm:higherrank} when
$\mathcal{F}$ is the structure sheaf $\Ok_Z$ of an analytic subspace $Z\subset X$ by
comparing $(E,\varphi)$ to a certain universal free resolution due to
Scheja and Storch, \cite{SS}, and Eisenbud, Riemenschneider, and
Schreyer, \cite{ERS}.
That proof should be possible to modify to the setting of a general
$\mathcal F$. However, we give a simpler proof using induction
over a filtration of
$\mathcal F$, see Section \ref{piga}.
Note that in \cite{LW}, the assumption that the connections on $E_0,\dots,E_N$ should be 
$(1,0)$ is missing, see the comment before Lemma~\ref{lmaindep} below.

If $(E,\varphi)$ is the Koszul complex of a tuple of holomorphic
functions $f_1,\ldots, f_m$, then the coefficients of $R$ are the
so-called \emph{Bochner-Martinelli residue currents} introduced by
Passare, Tsikh, and Yger \cite{PTY}, and further developed by
Andersson, \cite{AndIdeals}, see Section ~\ref{ssect:koszul}.
In particular, if $m=p:=\codim Z(f)$, where $Z(f)=\{f_1=\cdots = f_m=0\}$,
(the coefficient of) the only
nonvanishing component $R^0_p$ coincides with the classical
\emph{Coleff-Herrera
  product}
$\dbar (1/f_p)\wedge \cdots \wedge \dbar (1/f_1)$,
introduced by Coleff and Herrera in \cite{CH}, see
\cite{PTY}*{Theorem~4.1} and \cite{AndCH}*{Corollary~3.2}.
In this case \eqref{eq:main} reads
\begin{equation}\label{chpl}
    \frac{1}{(2\pi i)^p}\dbar \frac{1}{f_p} \wedge \cdots \wedge
    \dbar \frac{1}{f_1} \wedge df_1 \wedge \cdots \wedge df_p=[Z].
\end{equation}
This generalization of the Poincar\'e-Lelong formula \eqref{poin} was proved by
Coleff and Herrera
\cite[Section~3.6]{CH}.
If $m>p$, then $[E]=0$ and we can give an alternative proof of Theorem
~\ref{thm:main} by explicitly computing the left-hand side of
\eqref{eq:main}, see Section \ref{koszulsection}.
Since both sides in \eqref{eq:main} are alternating sums,
it would be a natural guess that the terms at respective levels in the
sums coincide. However, this in not true in general and the Koszul
complex provides a counterexample, see Example \ref{strumpa}.

There are various  other special cases of Theorem \ref{thm:higherrank}
and related results
in the literature, see, e.g., the introduction in \cite{LW}.
There are also related cohomological results by Lejeune-Jalabert
and Lejeune-Jalabert-Ang\'eniol, \cite{LJ1, ALJ}. Given a free resolution $(E,\varphi)$ of
$\Ok_{Z,z}$, where $Z$ is a Cohen-Macaulay analytic
space,  Lejeune-Jalabert, \cite{LJ1}, constructed a generalization of the Grothendieck
residue pairing, which can be seen as a cohomological version of
$R^E$, and proved that the fundamental class of $Z$ at $z$ then is
represented by
$D\varphi_1\cdots D\varphi_p$.
In \cite{ALJ} this construction was extended to a residue pairing
associated with a more general complex of free $\Ok_z$-modules and a
cohomological version, \cite[Theorem~I.8.2.2.3]{ALJ}, of Theorem
\ref{thm:main} was given.

In Section \ref{ickeren} we discuss possible extensions of our results
to the case when the homology groups $\mathcal H_\ell(E)$ do not have
pure dimension or are not of the same dimension. In particular,
we present a version of Theorem
\ref{thm:higherrank} for a general, not necessarily pure dimensional,
coherent sheaf $\mathcal F$, generalizing \cite[Theorem~1.5]{LW}.

\section{Preliminaries}

Throughout this article, $(E,\varphi)$ will be a
complex \eqref{komplexet},
where the $E_k$ are either vector bundles on $X$ or germs of free $\Ok$-modules, where $\Ok = \Ok_x = \Ok_{X,x}$
is the ring of germs of holomorphic functions at some $x \in X$.
We will always assume that $E_k = 0$ for $k < 0$ and $k > N$.
Since a complex $(E, \varphi)$ of $\Ok$-modules can be extended to
a vector bundle complex in a neighborhood of $x$ it makes
sense to equip it with hermitian metrics, and thus to talk about a
hermitian complex of $\Ok$-modules.

We let $\E$ and $\E^\bullet$ be the sheaves of smooth functions and
forms, respectively, on $X$. Given a vector bundle $E\to X$ we let
$\E^\bullet(E)=\E^\bullet\otimes \E(E)$ denote the sheaf of
form-valued sections.

\subsection{Signs and superstructure} \label{ssect:super}

As in \cite{AW1}, we will consider the complex $(E,\varphi)$ to be equipped with
a so-called superstructure, i.e., a $\mathbb{Z}_2$-grading, which
splits $E=\oplus E_k$ into odd and even
elements $E^+$ and $E^-$, where $E^+ = \oplus E_{2k}$ and $E^- = \oplus E_{2k+1}$.
Also $\End E$ gets a superstructure by letting the even elements be
the endomorphisms preserving the degree, and the odd elements the endomorphisms switching degrees.

This superstructure affects how form- and current-valued endomorphisms act.
Assume that $\alpha=\omega\otimes \gamma$ is a section of $\E^\bullet
(\End E)$, where $\gamma$ is a holomorphic section of
$\Hom(E_\ell,E_k)$, and $\omega$ is a smooth form of degree $m$.
Then we let $\deg_f \alpha = m$ and $\deg_e \alpha = k-\ell$ denote
the \emph{form} and \emph{endomorphism} degrees, respectively, of $\alpha$.
The \emph{total} degree is $\deg \alpha = \deg_f \alpha + \deg_e \alpha$.
The following formulas, which can be found in \cite{LW}, will be
important to get the signs right in the proofs of the main results.
Assume that $\alpha=\omega\otimes \gamma$ and $\alpha'=\omega'\otimes
\gamma'$ are sections of $\E^\bullet (\End E)$, where $\omega,
\omega'$ are sections of $\E^\bullet$ and $\gamma,\gamma'$ are sections of
$\End E$.
Due to how form-valued endomorphisms are defined to act on form-valued sections, one obtains the following
composition of form-valued endomorphisms, \cite{LW}*{equation (2.2)},
\begin{equation}\label{jord}
\alpha \alpha'=(-1)^{(\deg_e\alpha)(\deg_f\alpha')} \omega\wedge
\omega'\otimes \gamma \gamma'.
\end{equation}
We have the following formula for the trace, see
\cite{LW}*{equation (2.14)},
\begin{equation}\label{banal}
\tr(\alpha\alpha')=(-1)^{(\deg\alpha)(\deg\alpha')-(\deg_e\alpha)(\deg_e\alpha')}\tr(\alpha'\alpha).
\end{equation}

If the bundles $E_0,\ldots,E_N$ are equipped with connections
$D_{E_i}$, there is an induced connection
$D_E := \oplus D_{E_i}$ on $E$, which in turn induces a connection $D_{\End}$
on $\End E$, that takes the superstructure into account, through
\begin{equation}\label{DEnd}
D_{\End} \alpha := D_E \circ \alpha - (-1)^{\deg \alpha} \alpha \circ D_E.
\end{equation}
This connection satisfies the following Leibniz rule, \cite{LW}*{equation (2.4)},
\begin{equation}\label{bord}
D_{\End} (\alpha \alpha')=D_{\End} \alpha \alpha' + (-1)^{\deg \alpha}
\alpha D_{\End} \alpha'.
\end{equation}
To simplify notation, we will drop the subscript $\End$ and simply denote this connection by $D$.
All of the above formulas hold also when $\alpha$ and $\alpha'$ are current-valued instead of form-valued,
as long as the involved products of currents are well-defined.

Since $\varphi_{m}\varphi_{m+1}=0$ and the $\varphi_j$ have odd degree, by the Leibniz rule, $\varphi_m D\varphi_{m+1} = D\varphi_m \varphi_{m+1}$,
and using this repeatedly, we get that
    \begin{equation} \label{dal}
        D\varphi_{\ell} \cdots D\varphi_{k-1} \varphi_k =
        \varphi_\ell D\varphi_{\ell+1} \cdots D\varphi_k
    \end{equation}
for all $\ell<k$.

The following result is a slight generalization of
\cite{LW}*{Lemma~4.4}, that follows by the same arguments.

\begin{lma}\label{sniken}
    Let $p$ be fixed. Assume that $(E,\varphi)$ and $(G,\eta)$ are
    complexes of vector bundles
    and that $b : (E,\varphi) \to (G,\eta)$ is a morphism of complexes. Let $D$ be the connection on
    $\End (E\oplus G)$ induced by arbitrary connections on $E_\ell,\dots,E_{\ell+p}$ and $G_\ell,\dots,G_{\ell+p}$,
    and let\footnote{Here $D\eta_{\ell+1} \cdots D\eta_j$ and $D\varphi_{j+1}\cdots D\varphi_{\ell+p-1}$ are to be interpreted as $1$
    if $j = \ell$ and $j=\ell+p-1$, respectively.}
\begin{equation*}
    \begin{gathered}
\delta_\ell:=\sum_{j=\ell}^{\ell+p-1} D\eta_{\ell+1} \cdots D\eta_j D b_jD\varphi_{j+1}\cdots D\varphi_{\ell+p-1}
,\\
\alpha_\ell := \eta_{\ell+1}\delta_{\ell+1}, \quad \beta_\ell := \delta_\ell\varphi_{\ell+p}, \text{ and }
\gamma_\ell :=b_\ell D\varphi_{\ell+1}\cdots D\varphi_{\ell+p}.
    \end{gathered}
\end{equation*}
Then
\begin{equation*}
D\eta_{\ell+1} D\eta_{\ell+2}\cdots D\eta_{\ell+p}b_{\ell+p}=
 \alpha_\ell + \beta_\ell
  + \gamma_\ell.
\end{equation*}
\end{lma}



Given a complex $(E, \varphi)$, let $(\widetilde E, \tilde \varphi)$ be the complex where the signs are reversed, i.e., let $\widetilde E_k$ be
$E_k$ but with opposite sign and let $\tilde \varphi_k$ be the mapping
$\widetilde E_k\to \widetilde E_{k-1}$ induced by $\varphi_k$. Note
that $\tilde \varphi_k$ is odd.
More generally, for any section $\alpha$ of $\End E$ or $\E^\bullet
(\End E)$, let $\tilde \alpha$ denote the corresponding section of
$\End \widetilde E$ or $\E^\bullet
(\End \widetilde E)$, respectively.
Note that if $\alpha=\omega\otimes \gamma$ is a section of $\E^\bullet(\End
E)$, then $\tilde \alpha=\omega \otimes \tilde \gamma$.

Next, let $\varepsilon:E_k\to \widetilde E_k$ be the map induced by
the identity on $E_k$. Note that $\varepsilon$ is an odd mapping.
If $\gamma$ is a section of $\Hom (E_k,E_{k-1})$, then
\begin{equation}\label{pluto}
\varepsilon\gamma=\tilde\gamma\varepsilon.
\end{equation}
If $\alpha=\omega\otimes \gamma$ is a section of $\E^\bullet (\Hom
(E_k,E_{k-1}))$, then
\begin{equation*}
\varepsilon\alpha=
(-1)^{(\deg_e \varepsilon)(\deg_f\alpha)}\omega\otimes
\varepsilon\gamma
=
(-1)^{\deg_f\alpha}\omega\otimes
\varepsilon\gamma
=
(-1)^{\deg_f\alpha}\omega\otimes
\tilde\gamma\varepsilon;
\end{equation*}
here we have used \eqref{jord} for the first equality, that $\deg_e\varepsilon=1$ for the second
equality, and \eqref{pluto} for the third equality.
Moreover, by \eqref{jord},
\begin{equation*}
\tilde\alpha\varepsilon
=(-1)^{(\deg_e\tilde\alpha)(\deg_f\varepsilon)}\omega\otimes
\tilde\gamma\varepsilon=\omega\otimes
\tilde\gamma\varepsilon,
\end{equation*}
since $\deg_f\varepsilon=0$.
To conclude
\begin{equation}\label{mars}
\varepsilon\alpha=(-1)^{\deg_f\alpha}\tilde\alpha\varepsilon.
\end{equation}

\subsection{Residue currents and the comparison formula} \label{ssect:residue}

We will recall some properties from \cite{AW1} of the residue current $R=R^E$
associated with a hermitian complex $(E,\varphi)$, cf.\ the
introduction.
The part $R^\ell_k=(R^E)^\ell_k$ that takes values in
$\Hom(E_\ell,E_k)$ is a $(0,k-\ell)$-current when $\ell < k$ and
$R^\ell_k = 0$ otherwise.
For us, a key property of the current $R^E$ is that it is
$\nabla_\End$-closed, which means that
\begin{equation} \label{eq:nabla1}
    \varphi_{k+1} R^\ell_{k+1} - R_k^{\ell-1} \varphi_\ell - \dbar
    R_k^\ell = 0,
\end{equation}
 for each $\ell,k$, see \cite[Section~2]{AW1}.

The residue currents $R^E$ are examples of so-called \emph{pseudomeromorphic
currents}, introduced in \cite{AW2}. Another important
example 
is currents of integration 
along subvarieties $Z\subset X$,
as follows, e.g., from \cite[Theorem~1.1]{ALelong}.
The sheaf of pseudomeromorphic currents is closed under multiplication
by smooth forms.
Moreover pseudomeromorphic currents share some properties with normal currents, and in particular they satisfy the following \emph{dimension principle},
\cite{AW2}*{Corollary~2.4}:

\begin{prop}
    Let $T$ be a pseudomeromorphic $(*,p)$-current on $X$, and assume
    that $T$ has support on
    a subvariety $Z\subset X$ of $\codim Z > p$. Then $T = 0$.
\end{prop}

If $(E,\varphi)$ is pointwise exact outside a subvariety $Z$ of
codimension $p$ and $k-\ell < p$, since $R$ has support on $Z$,
it follows from the dimension principle that $R^\ell_k = 0$. Then \eqref{eq:nabla1} becomes
\begin{equation} \label{eq:nabla2}
    \varphi_{k+1} R^\ell_{k+1} = R_k^{\ell-1} \varphi_\ell.
\end{equation}
In the special case when $\ell = 0$ and $k=p-1$, since $E_{-1} = 0$, \eqref{eq:nabla2} gives that
\begin{equation}\label{boden}
\varphi_pR_p^0=0.
\end{equation}

If the sheaf complex $(E,\varphi)$ is exact except at $E_0$,
then $R^\ell_k=0$ for $\ell \geq 1$, see \cite[Theorem~3.1]{AW1}. We
can then write without ambiguity $R_k=R^E_k$ for
$R_k^0=(R^E)_k^0$. In
this case, for $\ell=1$ and $k=p$, \eqref{eq:nabla2} reads
\begin{equation}\label{havsbad}
    R_p\varphi_1=0.
\end{equation}

Given a morphism $a : (F,\psi) \to (E,\varphi)$ of complexes of free
$\Ok$-modules or vector bundles, the \emph{comparison formula} from \cite{LComp}
relates the associated residue currents $R^E$ and $R^F$. We begin by recalling an important situation when
one can construct such a morphism, see for example \cite{LComp}*{Proposition~3.1}.
In this result, it is crucial that $(F,\psi)$ and $(E,\varphi)$ are
complexes of free $\Ok$-modules; the corresponding statement would not necessarily be true globally if they were instead complexes
of vector bundles over ~$X$.

\begin{prop} \label{propcomplexcomparison}
    Let $\alpha : A' \to A$ be a homomorphism of $\Ok$-modules, let $(F,\psi)$ be a complex of
    free $\Ok$-modules with $\coker \psi_1 \stackrel{\cong}\rightarrow A'$, and let $(E,\varphi)$ be a free resolution of $A$.
    Then, there exists a morphism $a : (F,\psi) \to (E,\varphi)$ of complexes which extends ~$\alpha$.
\end{prop}

Here, we say that $a$ \emph{extends} $\alpha$ if the induced map $A'
\stackrel{\cong}{\rightarrow} \coker \psi_1
\stackrel{(a_0)_*}{\longrightarrow} \coker \varphi_1
\stackrel{\cong}{\rightarrow} A$ equals $\alpha$.
The comparison formula in its most general form,
\cite{LComp}*{equation~(3.4)}, states that for
$k > \ell$, there exist pseudomeromorphic $(0,k-\ell-1)$-currents
$M^\ell_k$ with values in $\Hom(F_\ell, E_k)$ and support on the union
$Z$ of the sets where
$(E,\varphi)$ and $(F,\psi)$ are not pointwise exact, such that
\begin{equation*}
    (R^E)^\ell_k a_\ell = a_k (R^F)^\ell_k + \varphi_{k+1} M^\ell_{k+1} + M_k^{\ell-1} \psi_\ell - \dbar M_k^\ell.
\end{equation*}
Here, $M^{\ell-1}_k$ is to be interpreted as $0$ if $\ell=0$. In all the cases we consider in this article,
we have that $k-\ell \leq \codim Z$.
Then it follows from the dimension principle that $M^\ell_k$ vanishes, since it is a $(0,k-\ell-1)$-current
with support on $Z$ and the comparison formula becomes
\begin{equation}\label{cirkus1}
    (R^E)^\ell_k a_\ell = a_k (R^F)^\ell_k + \varphi_{k+1} M^\ell_{k+1} + M_k^{\ell-1} \psi_\ell.
\end{equation}
Since $M^\ell_k$ takes values in $\Hom(F_\ell,E_k)$ it follows that if
$F_{\ell-1} = 0$, then $M_k^{\ell-1}=0$ and \eqref{cirkus1} reads
\begin{equation}\label{cirkus2}
    (R^E)^\ell_k a_\ell = a_k (R^F)^\ell_k + \varphi_{k+1} M^\ell_{k+1},
\end{equation}
and if in addition $E_{k+1} = 0$, then
\begin{equation}\label{cirkus3}
    (R^E)^\ell_k a_\ell = a_k (R^F)^\ell_k.
\end{equation}

The following result is a corrected version of \cite{LW}*{Lemma~4.1}. In the proof in \cite{LW}*{Lemma~4.1} it was used that the connections
are $(1,0)$-connections, i.e., that the $(0,1)$-part of the connections is $\dbar$, but we missed adding this assumption in the statement of the lemma,
and then consequently in all other results relying on this, i.e., Theorem~1.1, 1.2, 1.5, 6.1 and Lemma 4.2.

\begin{lma} \label{lmaindep}
    Let $M$ be a finitely generated $\Ok$-module of codimension $p$
      and let $(E,\varphi)$ and $(F,\psi)$ be hermitian free resolutions of $M$.
    Then,
    \begin{equation*} 
        \tr D\varphi_1 \cdots D\varphi_p R^E_p = \tr D\psi_1 \cdots D\psi_p R^F_p,
    \end{equation*}
    where $D$ is the connection on $\End(E\oplus F)$ induced by arbitrary
    $(1,0)$-connections on $E_0,\dots,E_p$ and $F_0,\dots,F_p$.
\end{lma}

\subsection{The Koszul complex, Coleff-Herrera products, and
  Bochner-Martinelli residue currents} \label{ssect:koszul}

Let $f = (f_1,\ldots,f_m)$ be a tuple of holomorphic functions on $X$
and let $(E,\varphi)$ be the Koszul complex of $f$, i.e., consider
$f$ as a section $f=\sum f_j e_j^*$ of a trivial rank $m$ bundle $F^*$
with a frame $e_1^*,\ldots, e_m^*$, let $E_j=\bigwedge^j F$, where $F$
is the dual bundle of $F^*$, and let
$\varphi_k=\delta_f$ be contraction with $f$.

The residue current $R$ associated with the Koszul complex
$(E,\varphi)$
equipped with hermitian metrics induced by a hermitian metric on $F^*$,
was introduced
and studied by Andersson in \cite{AndIdeals}.
For $E_k$, we have the frame $\{ e_I = e_{i_1} \wedge \dots \wedge
e_{i_k} \mid I = (i_1,\dots,i_k), 1 \leq i_1 < \dots < i_k \leq m \}$,
where $e_1,\ldots, e_m$ is the dual frame of $e_1^*,\ldots, e_m^*$,
and in particular, $e_\emptyset$ is a frame for $E_0$.
In this frame we can write $R^0_p = \sum R_I \wedge
e_I \wedge e_\emptyset^*$. To get the superstructure right, in
\cite{AndIdeals} it is convenient to consider the endomorphism-valued
currents and forms that appear in the construction of $R$ as sections
of the exterior algebra of $F^*\oplus T^*_{0,1}$, i.e., with the
convention that $dz_j \wedge e_k =-e_k\wedge dz_j$
etc. If the metric on $F^*$ is trivial, then the coefficients $R_I$
coincide with the Bochner-Martinelli residue currents from \cite{PTY}.

In the case when $m = p=\codim Z(f)$,
so that the ideal $\J(f)$ generated by $f$ is a complete intersection, then the Koszul complex
is a locally free resolution
of $\Ok_Z:=\Ok/\J(f)$. Since $R^\ell_k = 0$ for $k-\ell < p$ by the
dimension principle, in this case $R = R^0_p$ and $R^0_p$ consists of only one
component, $R_{\{1,\ldots,p\}} \wedge e_1\wedge \cdots \wedge e_p$, 
where  $R_{\{1,\ldots,p\}} =\dbar (1/f_p)\wedge\cdots\wedge \dbar (1/f_1)$,
cf.\ the introduction.
Moreover, $D\varphi_j$ is contraction with $\sum df_j \wedge
e_j^*$ and it follows that
\begin{equation}\label{disputation}
 D\varphi_1 \cdots D\varphi_p = p! df_1 \w \cdots \w df_p\wedge
 e_p^*\wedge\cdots\wedge e_1^*.
\end{equation}
Therefore, the generalized Poincar\'e-Lelong formula \eqref{chpl} by Coleff
and Herrera can be rewritten as
    \begin{equation}\label{elefant}
        \frac{1}{(2\pi i)^p p!} D\varphi_1 \cdots D\varphi_p R^E_p = [Z].
    \end{equation}
For an explanation of the signs when going from endomorphism-valued
currents to scalar-valued currents, see
\cite[Section~2.5]{LW} and also Section~\ref{ssect:super} above.

\subsection{The mapping cone of a morphism of complexes}

Let $c : (L,\lambda) \to (K,\kappa)$ be a morphism of complexes. The \emph{mapping cone} of $c$ is the complex
$(C,\mu)$ given by $C_k = K_k \oplus \widetilde L_{k-1}$ for $k \geq 1$ and $C_0 = K_0$, with
\begin{equation*}
    \mu_k = \left[ \begin{array}{cc} -\kappa_k  & c_{k-1}
                                                  \varepsilon^{-1} \\
                     0 & \tilde{\lambda}_{k-1} \end{array}\right]  \text{ for $k \geq 2$ and }
    \mu_1 = \left[ \begin{array}{cc} -\kappa_1 & c_0 \varepsilon^{-1} \end{array}\right].
\end{equation*}
Here, the bundles and morphisms take into account the signs and superstructure from Section \ref{ssect:super}.
Let
\begin{equation} \label{eq:thetaMC}
    \theta_k : K_k \to C_k \text{, }
    \theta_k = \left[ \begin{array}{c} (-1)^k \Id_{K_k} \\ 0 \end{array}\right] \text{ for $k \geq 1$, }
    \theta_0 = \left[ \begin{array}{c} \Id_{K_0} \end{array}\right]
\end{equation}
and
\begin{equation} \label{eq:varthetaMC}
    \vartheta_k : C_{k+1} \to L_k \text{, }
    \vartheta_k =  \left[\begin{array}{cc} 0 &
                                               \varepsilon^{-1} \end{array} \right] \text{ for $k \geq 0$}.
\end{equation}
Then $\theta : (K,\kappa) \to (C,\mu)$ and $\vartheta : (C,\mu) \to (L,\lambda)$ are morphisms of complexes
(the latter of degree $-1$).
From this construction one obtains, cf., e.g., \cite{Weibel}*{Chapter 1.5}, an induced long exact sequence
\begin{equation} \label{eq:les-mc}
    \cdots \to H_{k+1}(C) \stackrel{\vartheta_k}{\longrightarrow} H_k(L) \stackrel{c_k}{\longrightarrow} H_k(K) \stackrel{\theta_k}{\longrightarrow} H_k(C) \stackrel{\vartheta_{k-1}}{\longrightarrow} H_{k-1}(L) \to \cdots.
\end{equation}

\begin{prop} \label{prop:mc-ses}
    Let $$0 \to A' \stackrel{\alpha}{\to} A \stackrel{\beta}{\to} A'' \to 0$$
    be a short exact sequence of $\Ok$-modules.
Assume that $(E,\varphi)$ and $(F,\psi)$ are free resolutions of $A$ and $A'$,
respectively, and that
$a : (F,\psi) \to (E,\varphi)$ is a morphism of complexes extending
$\alpha$.
    Let $(G,\eta)$ be the mapping cone of $a$ and let $b : (E,\varphi) \to (G,\eta)$ be the
    morphism $\theta$ as defined by \eqref{eq:thetaMC}.
    Then there is an isomorphism $H_0(G) \stackrel{\cong}{\to} A''$, which makes
    $(G,\eta)$ a free resolution of $A''$ and such that $b$ extends $\beta$.
\end{prop}

\begin{proof}
    By \eqref{eq:les-mc}, we obtain that $H_\ell(G) = 0$ for $\ell \neq 0,1$ and the exact sequence
    \begin{equation*}
        0 \to H_1(G) \to H_0(F) \to H_0(E) \to H_0(G) \to 0.
    \end{equation*}
    The morphism $H_0(F) \to H_0(E)$ equals the morphism
    $A' \stackrel{\alpha}{\to} A$ which is injective, so $H_1(G) = 0$.
    Thus $(G,\eta)$ is a free resolution of $A/(\im \alpha)$. Since $\beta$ gives
    an isomorphism $A/(\im \alpha) = A/(\ker \beta) \stackrel{\cong}{\to} \im \beta = A''$,
    $(G,\eta)$ is a free resolution of $A''$.
    By construction, $b$ extends the morphism $A \to A/(\im \alpha)$.
\end{proof}

\section{Proof of Theorem~\ref{thm:higherrank}}\label{piga}

The proof of Theorem~\ref{thm:higherrank} is by induction.
The induction procedure is achieved through the following filtration
of a module, see, e.g., \cite{Bou}*{\S1.4, Th\'eor\`emes 1 and 2 and \S2.5, Remarque 1} or \cite{Stacks}*{Tag 00KY},
that is sometimes referred to as a \emph{prime filtration}.

\begin{prop} \label{prop:decomposition}
    Let $M$ be a finitely generated $\Ok$-module. Then there exists a sequence of submodules
    \begin{equation}\label{sommaren}
        0 = M_0 \subset M_1 \subset  \cdots \subset M_m = M
    \end{equation}
    such that
    \begin{equation} \label{eq:simpleQuotient}
        M_i/M_{i-1} \cong \Ok/P_i,
    \end{equation}
where $P_i \subseteq \Ok$ is a prime ideal contained in
$\supp M$ for $i=1,\ldots,m$.
    The minimal prime ideals $P_i$ (with respect to inclusion) appearing in \eqref{eq:simpleQuotient} are exactly
    the minimal associated primes of $M$, and each such minimal prime $P$ occurs exactly $\length_{\Ok_P} M_P$
    times.
\end{prop}

In general, also primes $P_i$ appear in \eqref{eq:simpleQuotient} that are not minimal primes of $M$, as in the following example.
If only the minimal primes of $M$ appear, then the filtration is said
to be a \emph{clean}, cf.\ \cite{Dress}.

\begin{ex}
    Let $\mathcal{J} = \mathcal{J}(xz,xw,yz,yw) \subseteq \Ok_{\C^4}$, which is the ideal generating the variety $\{ x = y = 0 \} \cup \{ z = w = 0 \}$.
    If we let
    \begin{equation*}
        \mathcal{I}_0 = \Ok, ~\mathcal{I}_1 = \mathcal{J}(x,y,w),
        ~\mathcal{I}_2 = \mathcal{J}(xz, y, w), ~\mathcal{I}_3 = \mathcal{J}(xz,xw,y), ~\mathcal{I}_4 = \mathcal{J},
    \end{equation*}
    then $M_j :=\Ok/\mathcal{I}_j$ for $j=1,\ldots, 4$, is a prime filtration of
    $M:=\Ok/\mathcal{J}$. Indeed, $M_j/M_{j-1}\cong
    \mathcal{I}_{j-1}/\mathcal{I}_{j} \cong \Ok/P_j$, where
    \begin{equation*}
        P_1  = \mathcal{J}(x,y,w), ~P_2 = \mathcal{J}(y,z, w), ~P_3
        \cong \mathcal{J}(x,y), ~P_4 \cong \mathcal{J}(z,w).
    \end{equation*}
    Note that $P_3$ and $P_4$ are the two (minimal) associated primes
    of $M$,
which have codimension $2$,
    while $P_1$ and $P_2$ are of codimension $3$ and contained in
    the support of $M$ but not associated primes of $M$.
\end{ex}

\begin{cor} \label{cor:clean}
    Let $\mathcal{F}$ be a coherent sheaf of codimension $p$
    and let $z_0 \in \supp \mathcal{F}$.
    For $z$ in a neighborhood of $z_0$, outside a subvariety of
    positive codimension in $\supp \mathcal F$,
    $\mathcal{F}_z$ has a clean filtration where all the modules in the filtration have pure codimension ~$p$.
\end{cor}

\begin{proof}
Take a filtration of $\mathcal{F}_{z_0}$ as in Proposition
\ref{prop:decomposition}
and choose a neighborhood $z_0\in \mathcal U\subset X$ such that
all $M_i$ are defined in $\mathcal U$. Moreover let $W$ be the union of the varieties of the $P_i$
that have codimension $\geq p+1$.
Take $z\in \mathcal U\setminus W$. For each $i$, at $z$, either $M_{i+1}=M_i$ or
$M_{i+1}/M_i\cong \Ok/P_j$ for some $j$, where $P_j$ is an associated prime of $\mathcal{F}_z$ of codimension $p$. Thus if we remove the $M_i$ such that
$M_{i+1}=M_i$ we are left with a clean filtration of $\mathcal F_z$.
Since the sequence $0 \subset M_0 \subset \cdots \subset M_k$ gives a filtration of $M_k$,
by Proposition~\ref{prop:decomposition}, the only minimal primes of $M_k$ are $P_i$ for $i=1,\ldots,k$,
and thus $M_k$ has pure codimension $p$ for $k=1,\ldots,m$.
\end{proof}


\begin{lma} \label{lma:basiccase}
Let $P\subset \Ok$ be a prime ideal of codimension $p$ and let
$(E,\varphi)$ be a hermitian free resolution of $\Ok/P$. 
    Then
    \begin{equation}\label{eq:basiccase}
        \frac{1}{(2\pi i)^p p!}\tr D\varphi_1 \cdots D\varphi_p R_p^E = [\Ok/P].
    \end{equation}
\end{lma}

\begin{proof}
Since both sides of \eqref{eq:basiccase} are pseudomeromorphic
$(p,p)$-currents with support on the variety $Z$ of $P$,
it is by the dimension principle enough to prove that
\eqref{eq:basiccase} holds locally on $Z_{\rm reg}$.
We may thus assume that we have local coordinates $(z_1,\ldots,z_n)$ such that $Z = \{ z_1 = \cdots = z_p = 0 \}$.
Since the left-hand side of \eqref{eq:basiccase} is independent of the
choice of locally free resolution $(E,\varphi)$ by Lemma~\ref{lmaindep}, we can
assume that $(E,\varphi)$ is the Koszul complex of $z_1,\ldots, z_p$.
In this case, it follows from \eqref{elefant} 
that the left-hand side of \eqref{eq:basiccase} equals $[z_1 = \dots = z_p = 0 ] = [Z]=[\Ok/P]$.
\end{proof}

\begin{prop} \label{prop:Rses}
    Let
    \begin{equation*} 
        0 \to A' \to A \to A'' \to 0
    \end{equation*}
    be an exact sequence of $\Ok$-modules of codimension $p$, and let
    $(E,\varphi)$, $(F,\psi)$, and $(G,\eta)$ be hermitian free resolutions of $A$, $A'$, and $A''$, respectively.
    Then
    \begin{equation} \label{eq:curSes}
        \tr D\varphi_1\cdots D\varphi_p R^E_p = \tr D\psi_1 \cdots D\psi_p R^F_p + \tr D\eta_1 \cdots D\eta_p R^G_p.
    \end{equation}
\end{prop}

\begin{proof}
    By the dimension principle it is enough to prove \eqref{eq:curSes}
    outside a subvariety of codimension $p+1$, since all currents in the equation are pseudomeromorphic $(p,p)$-currents.
    Since a module of codimension $p$ is Cohen-Macaulay outside a subvariety of codimension $\geq p+1$,
    we may thus assume that $A, A', A''$ are all Cohen-Macaulay.
    By Lemma ~\ref{lmaindep} we may assume that
    $(E,\varphi)$, $ (F,\psi)$, and $(G,\eta)$ are any free
    resolutions of $A$, $A'$, and $A''$, respectively. In particular, we
    may assume that  $(E,\varphi)$ and $(F,\psi)$ have length
    $p$. Moreover, by Propositions ~\ref{propcomplexcomparison} and
    ~\ref{prop:mc-ses}  we may assume that $(G,\eta)$ is the mapping
    cone of a morphism $a : (F,\psi) \to (E,\varphi)$ that extends the
    inclusion $A'\to A$.
    Note that $(G, \eta)$ then has length $p+1$. Let $b: (E,\varphi) \to (G,\eta)$ be the morphism
    $\theta$ as defined in \eqref{eq:thetaMC}.
Since $b_0$ is invertible, one gets by the comparison formula, \eqref{cirkus2}, that
    \begin{equation} \label{eq:trE}
    D\eta_1 \cdots D\eta_p R^G_p = D\eta_1 \cdots D\eta_p b_p R^E_p b_0^{-1} + D\eta_1 \cdots D\eta_p\eta_{p+1} M^0_{p+1} b_0^{-1}
    =:  W_1 + W_2.
    \end{equation}

    By Lemma~\ref{sniken}, $W_1 = (\alpha_0 + \beta_0 + \gamma_0) R^E_p b_0^{-1},$
    where $\alpha_0,\beta_0,\gamma_0$ are as in the lemma.
    Since, by Lemma~\ref{lmaindep}, $\tr D\eta_1 \cdots D\eta_p R^G_p$
    is independent of the choice of connections on $G_0,\dots,G_p$,
    we may assume that the connection on $G_j$
    is such that it respects the direct sum $G_j = E_j\oplus
    \widetilde F_{j-1}$ for $j \geq 1$  and that it coincides with $D_{E_j}$ on $E_j \subseteq G_j$ for $j \geq 0$.
    With this connection $Db_j = 0$, so $\alpha_0 = \beta_0 = 0$, and, using
    \eqref{banal}, we conclude that
    \begin{equation} \label{eq:trEfirst}
        \tr W_1 = \tr \gamma_0 R^E_p b_0^{-1} = \tr b_0 D\varphi_1 \cdots
        D\varphi_p R^E_p b_0^{-1}= \tr D\varphi_1 \cdots D\varphi_p R^E_p.
    \end{equation}

    We now consider the term $\tr W_2$.
Since $A''$ is Cohen-Macaulay of codimension $p$, $A$ has codimension $p$, and $(G,\eta)$ is a
free resolution of $A''$ of length $p+1$, by \cite{LComp}*{Lemma~3.3 and equation
  (3.11)}, $M^0_{p+1} = -\sigma^G_{p+1} b_p R^E_p$, where
$\sigma^G_{p+1}$ is a smooth
    $\Hom(G_p,G_{p+1})$-valued morphism, such that
\begin{equation}\label{nate}
\sigma_{p+1}^G\eta_{p+1}=\Id_{G_{p+1}}.
\end{equation}
Therefore, in view of \eqref{dal} and \eqref{banal},
    \begin{equation*}
        \tr W_2 = - \tr D\eta_2 \cdots D\eta_{p+1} \sigma^G_{p+1} b_p R^E_p b_0^{-1} \eta_1.
    \end{equation*}
    Using that $b_0=\Id_{E_0}$, that $\eta_1 = \left[\begin{array}{cc} -\varphi_1 & a_0 \varepsilon^{-1} \end{array}\right]$, \eqref{havsbad}, and the comparison formula, \eqref{cirkus3}, for $a : (F,\psi) \to (E,\varphi)$, we get that
    \begin{equation*}
        R^E_p b_0^{-1}\eta_1 = R^E_p \left[\begin{array}{cc} 0 & a_0 \varepsilon^{-1} \end{array} \right] = \left[\begin{array}{cc} 0 & a_p R^F_p \varepsilon^{-1}  \end{array}\right].
    \end{equation*}
Note that
\begin{equation*}
    \eta_{p+1} = \left[ \begin{array}{c} a_{p} \varepsilon^{-1} \\
                          \tilde{\psi}_{p} \end{array}\right] \text{
                        and thus }
D\eta_2 \cdots D\eta_{p+1} = \left[\begin{array}{c} *  \\
                                     D\tilde\psi_1 \cdots
                                     D\tilde\psi_p \end{array}\right].
\end{equation*}
It follows that
    \begin{multline}\label{natalitet}
\tr W_2
            = -\tr \left[\begin{array}{c} *  \\ D\tilde\psi_1 \cdots D\tilde\psi_p \end{array}\right] \sigma^G_{p+1} b_p \left[\begin{array}{cc} 0 & a_p R^F_p \varepsilon^{-1} \end{array}\right] \\
                = - \tr D\tilde\psi_1 \cdots D\tilde\psi_p
                \sigma^G_{p+1} b_p a_p R^F_p \varepsilon^{-1}.
    \end{multline}
Moreover,
\begin{equation*}
b_p a_p = \left[\begin{array}{c} (-1)^p a_p \\ 0 \end{array}\right] =
(-1)^p \left[\begin{array}{c} a_p \varepsilon^{-1} \\
        0 \end{array}\right]\varepsilon = (-1)^p \left( \eta_{p+1} -
      \left[\begin{array}{c} 0 \\ \tilde\psi_p \end{array}\right]
    \right)\varepsilon.
\end{equation*}
In view of \eqref{pluto} and \eqref{boden}, note that $\tilde\psi_p\varepsilon R^F_p = \varepsilon \psi_p   R^F_p = 0$. Therefore
    \begin{equation*}
        D\tilde\psi_1 \cdots D\tilde\psi_p \sigma^G_{p+1} b_p a_p R^F_p
        = (-1)^p D\tilde\psi_1 \cdots D\tilde\psi_p \varepsilon R^F_p
= \varepsilon D\psi_1 \cdots D\psi_p  R^F_p,
    \end{equation*}
cf.\ \eqref{nate} and \eqref{mars}.
   Plugging this into \eqref{natalitet} and using
    \eqref{banal}, we get
    \begin{equation} \label{eq:trEsecond}
        \tr W_2 =
- \tr \varepsilon D\psi_1 \cdots D\psi_p  R^F_p\varepsilon^{-1} = - \tr D\psi_1 \cdots
        D\psi_p  R^F_p.
    \end{equation}
    We thus conclude that \eqref{eq:curSes} holds by combining \eqref{eq:trE}, \eqref{eq:trEfirst}, and \eqref{eq:trEsecond}.
\end{proof}

\begin{proof}[Proof of Theorem~\ref{thm:higherrank}]
It is enough to prove \eqref{nordbotten} locally and by the dimension principle, since both sides of \eqref{nordbotten}
are pseudomeromorphic currents of bidegree $(p,p)$, it is enough to prove \eqref{nordbotten} outside
a subvariety of $Z := \supp \mathcal{F}$ of positive codimension.
By Corollary ~\ref{cor:clean} we may thus assume that we are at a point $z \in Z$
    such that $\mathcal{F}_z$ has a clean filtration
    \eqref{sommaren}, where each $M_i$ is of pure codimension $p$; in particular each $P_i$ is of  codimension ~$p$.

  Since $\mathcal F_z = M_m$, we may prove the theorem by
  proving it for $M_i$ by induction over $i$.
The basic case $i=1$ follows by Lemma~\ref{lma:basiccase}.
Next assume that Theorem \ref{thm:higherrank} holds for $M_i$. Consider the short exact sequence
    \begin{equation*}
        0 \to M_i \to M_{i+1} \to \Ok/P_{i+1} \to 0
    \end{equation*}
and assume that $(F,\psi)$, $(E,\varphi)$, and $(G,\eta)$ are
hermitian free resolutions of $M_i$, $M_{i+1}$, and $\Ok/P_{i+1}$,
respectively.
Then
\begin{multline*}
\frac{1}{(2\pi i)^p p!} \tr D\varphi_1\cdots D\varphi_p R^E =\\
\frac{1}{(2\pi i)^p p!} \tr D\psi_1 \cdots D\psi_p R^F_p +
\frac{1}{(2\pi i)^p p!} \tr D\eta_1 \cdots D\eta_p R^G
=\\
[M_i]+[\Ok/P_{i+1}]=[M_{i+1}],
\end{multline*}
where we have used Proposition ~\ref{prop:Rses} for the first
equality, and the induction hypothesis and Lemma \ref{lma:basiccase} for
the second. For the last equality, we use the fact that if
we have a short exact sequence of $\Ok$-modules, $0\to A'\to A\to A'' \to 0$, then
\begin{equation*}
[A]=[A']+[A''].
\end{equation*}
\end{proof}

\section{Proof of Theorem~\ref{thm:main}}\label{stiga}

The proof of Theorem~\ref{thm:main} is by induction over the number of nonvanishing
homology groups. In order to achieve a complex with one less nonvanishing homology group,
we will use the following lemma.

\begin{lma} \label{lma:bigDiagram}
    Let $(E,\varphi)$ be a complex of free $\Ok_{X,x}$-modules of length $k+p$, where $k \geq 1$, such that $H_\ell(E) = 0$ for $\ell > k$.
    Then there exists a neighborhood $x\in\mathcal U\subset X$ and a
    subvariety $W\subset\mathcal U$ of codimension $\geq p+1$ such that for
    $y\in \mathcal U\setminus W$ one can find
    free $\Ok_{X,y}$-modules $G_\ell$ and morphisms $\eta_\ell$ and $b_\ell$ for $\ell=k+1,\dots,k+p-1$ such that the diagram
    \begin{equation} \label{eq:bigDiagram}
\xymatrixcolsep{5mm}
\xymatrix{
   & 0 \ar[r] & G_{k+p-1} \ar[r]
& G_{k+p-2} \ar[r] & \cdots \ar[r] &
   G_{k+1} \ar[r] &  E_k \ar[r] &  E_{k-1} \ar[r] & \cdots\ar[r] &
   E_0 \ar[r]  & 0\\
   0 \ar[r] & E_{k+p} \ar[r] \ar[u] & E_{k+p-1}
   \ar[r] \ar[u]^{b_{k+p-1}} & E_{k+p-2}
   \ar[r] \ar[u]^{b_{k+p-2}} &  \cdots \ar[r] & E_{k+1}
   \ar[u]^{b_{k+1}}  \ar[r] & E_k \ar[r]
   \ar[u]^{\cong}& E_{k-1} \ar[r]\ar[u]^{\cong} & \cdots \ar[r] &
   E_0\ar[u]^{\cong} \ar[r]  & 0 \\
   0 \ar[r] & E_{k+p} \ar[r] \ar[u]^{\cong} & E_{k+p-1} \ar[r] \ar[u]^{\cong} & F_{k+p-2} \ar[r] \ar[u]^{a_{k+p-2}} & \cdots \ar[r] & F_{k+1} \ar[u]^{a_{k+1}} \ar[r] & \widetilde G_{k+1} \ar[u]^{a_k}  \ar[r]& 0 \ar[u] &
   }
\end{equation}
    is commutative and the rows $(G,\eta)$ and 
    $(F,\psi)$ are complexes,
and $a:(F,\psi)\to (E,\varphi)$ and $b : (E,\varphi)\to (G,\eta)$ are
morphisms of complexes.
Here 
    $F_\ell = \widetilde G_{\ell+1} \oplus E_\ell$ for $\ell=k+1,\ldots,k+p-2$ and
\begin{multline*}
    a_k = \left[ \begin{array}{c} b_k^{-1} \eta_{k+1} \varepsilon^{-1} \end{array} \right] \text{, }
    a_\ell = \left[\begin{array}{cc} 0 & \Id_{E_\ell} \end{array}\right] \text{ for $\ell=k+1,\dots,k+p-2$, } \\
    \psi_{k+p-1} = \left[\begin{array}{c} \varepsilon b_{k+p-1} \\ \varphi_{k+p-1} \end{array}\right] \text{, }
    \psi_{k+\ell} = \left[\begin{array}{cc} -\tilde\eta_{k+\ell+1} & \varepsilon b_{k+\ell} \\ 0 & \varphi_{k+\ell} \end{array}\right]
    \text{ for $\ell = 2,\dots,p-2$, and } \\
    \psi_{k+1} = \left[\begin{array}{cc} -\tilde\eta_{k+2} & \varepsilon b_{k+1} \end{array}\right].
\end{multline*}
    Moreover, the complexes in the rows have the following homology groups:
    \begin{equation}
    \begin{gathered} \label{eq:rowsHomology}
    \xymatrixcolsep{5mm}
    \xymatrixrowsep{5mm}
    \xymatrix{
    \cdots & 0 & 0              & H_{k-1}(E) & H_{k-2}(E) & \cdots & H_0(E) \\
    \cdots & 0 & H_k(E) & H_{k-1}(E) & H_{k-2}(E) & \cdots & H_0(E) \\
    \cdots & 0 & H_k(E) & 0                  & 0                  & \cdots & 0
    }
    \end{gathered}.
    \end{equation}
\end{lma}

The left-most part of the complex $(G,\eta)$ is a free resolution of $\coker \varphi_k$ and
the complex $(F,\psi)$ is essentially the mapping cone of $b : (E,\varphi) \to (G,\eta)$.

\begin{proof}
Let $\mathcal{F} = \coker \varphi_k$. 
Given any locally free resolution $(K,\kappa)$ of $\mathcal{F}$, there
are associated (germs of) subvarieties
$Z^K_k$ where $\kappa_k$ does not have maximal rank. By uniqueness of minimal free resolutions,
these sets are independent of the choice of resolution $(K,\kappa)$ and thus associated with $\mathcal{F}$,
and we may denote them by $Z^\mathcal{F}_k$ instead. Take $\mathcal U$
to 
be any neighborhood of $x$ where $(E,\varphi)$ is defined and $W$ to be $Z^\mathcal{F}_{p+1}$.
By the Buchsbaum-Eisenbud criterion $\codim W \geq p+1$, see \cite{Eis}*{Theorem~20.9}.

Assume that we are outside $W$, and take locally a free
resolution
\begin{equation*}
    0 \to K_{N_0} \stackrel{\kappa_{N_0}}{\longrightarrow} \cdots \to K_2
    \stackrel{\kappa_2}{\longrightarrow} E_k \stackrel{\varphi_k}{\longrightarrow} E_{k-1}
\end{equation*}
of $\mathcal F$.
Since we are outside $W$, $\im \kappa_{p+1}$ is free
so if we replace $K_p$ by $K_p/\im \kappa_{p+1}$, we can assume that the free resolution is of the form
\begin{equation*}
    0 \to K_p \stackrel{\kappa_p}{\longrightarrow} \cdots \to K_2 \stackrel{\kappa_2}\longrightarrow E_k \stackrel{\varphi_k}{\longrightarrow} E_{k-1}.
\end{equation*}
We let $G_{k+\ell-1} = K_\ell$ and $\eta_{k+\ell-1} = \kappa_\ell$ for $\ell=2,\dots,p$,
which then gives the complex in the top row of
\eqref{eq:bigDiagram}. This complex has the stated homology groups in the first row of \eqref{eq:rowsHomology}, since $(G,\eta)$ by construction is exact at levels $\geq k$.

Since the top row of \eqref{eq:bigDiagram} is exact at levels $\geq k$ and the modules in the middle row are free,
one can, locally, by a diagram chase inductively construct
$b_{k+1},\dots,b_{k+p-1}$ so that the diagram \eqref{eq:bigDiagram} commutes in the top two rows,
cf., e.g., the proof of \cite{Eis}*{Proposition A3.13}.

We now turn to the bottom row of \eqref{eq:bigDiagram}. Let $(C,\mu)$ be the mapping cone of the morphism
$b : (E,\varphi) \to (G,\eta)$ and let $\vartheta : (C,\mu) \to (E,\varphi)$ be the induced morphism of complexes
of degree $-1$ as defined by \eqref{eq:varthetaMC}.
Recall that
\begin{equation*}
    \mu_\ell = \left[ \begin{array}{cc} -\eta_\ell & b_{\ell-1} \varepsilon^{-1} \\ 0 & \tilde{\varphi}_{\ell-1} \end{array} \right].
\end{equation*}
On $C_\ell$ for $\ell=0,\ldots,k+1$, we do the change of basis given by the isomorphism
\begin{equation*}
    \alpha_\ell = \left[ \begin{array}{cc} \Id_{G_\ell} & 0 \\ \varepsilon b_{\ell-1}^{-1} \eta_{\ell} & \Id_{\widetilde E_{\ell-1}} \end{array} \right],
\end{equation*}
i.e., we replace $\mu_\ell$ by $\alpha_{\ell-1}^{-1}\mu_\ell
\alpha_\ell$ for $\ell=1,\ldots,k+1$, $\mu_{k+2}$ by
$\alpha_{k+1}^{-1}\mu_{k+2}$, and $\vartheta_\ell$ by
$\vartheta_{\ell} \alpha_{\ell+1}$ for $\ell=0,\ldots,k$.
Note that for these $\ell$, $b_{\ell-1}$ is the identity and thus invertible.
In this new basis, using that $\varepsilon b_{\ell-1}^{-1}\eta_\ell b_\ell \varepsilon^{-1} = \tilde{\varphi}_\ell$ for $\ell=1,\ldots,k+1$,
we get that
\begin{multline*}
    \mu_{k+2} = \left[ \begin{array}{cc} -\eta_{k+2} & b_{k+1} \varepsilon^{-1}  \\ 0 & 0 \end{array} \right],
    \mu_\ell = \left[ \begin{array}{cc} 0 & b_{\ell-1} \varepsilon^{-1} \\ 0 & 0 \end{array} \right] \text{ for $\ell=1,\ldots,k+1$, and } \\
    \vartheta_{\ell} = \left[ \begin{array}{cc} b_\ell^{-1}
                                \eta_{\ell+1} &
                                                \varepsilon^{-1} \end{array}
                                            \right]
\text{ for $\ell=0,\ldots,k$}.
\end{multline*}
Hence $(C,\mu)$ contains as summands the trivial complexes
\begin{equation} \label{eq:trivialParts}
    0 \to \widetilde E_{\ell-1} \stackrel{b_{\ell-1}\varepsilon^{-1}}{\longrightarrow} G_{\ell-1} \to 0 \text{ for $1 \leq \ell \leq k+1$.}
\end{equation}
We let $(F,\psi)$ be the complex $(C,\mu)$ where we use the new basis
as described above and remove the
trivial summands \eqref{eq:trivialParts}, and where we moreover shift
the degree by $1$ and change the signs so that $F_\ell=\widetilde
C_{\ell+1}$ and the morphisms are adjusted accordingly.
The morphisms $\mu$ and $\vartheta$ given by the mapping cone, adjusted accordingly, are then indeed the morphisms $\psi$ and $a$ as in the statement of the lemma.
Since $(b_\ell)_* : H_\ell(E) \to H_\ell(G)$ is an isomorphism for
$\ell=0,\ldots,k-1$, $H_\ell(G) = 0$ for $\ell \geq k$,
and $H_\ell(E)= 0$ for $\ell \geq k+1$, it follows from the long exact sequence \eqref{eq:les-mc} that $H_{\ell}(F) = H_{\ell+1}(C) = 0$ for $\ell \neq k$
and that $H_k(F) = H_{k+1}(C) \stackrel{\cong}{\to} H_k(E)$.
\end{proof}

For a complex $(E,\varphi)$ of length $N = k+p$, we will introduce the shorthand notation
\begin{equation*}
    (\tr D\varphi R^E)_p := \sum_{\ell=0}^k (-1)^\ell \tr D\varphi_{\ell+1} \cdots D\varphi_{\ell+p} (R^E)^{\ell}_{\ell+p}.
\end{equation*}

\begin{prop} \label{prop:inductionStep}
    Let $(E,\varphi)$, $(F,\psi)$, and $(G,\eta)$ be as in Lemma~\ref{lma:bigDiagram} and
    assume they are generically exact hermitian complexes.
    Then
    \begin{equation} \label{eq:inductionStep}
        (\tr D\varphi R^E)_p = (\tr D\psi R^F)_p + (\tr D\eta R^G)_p.
    \end{equation}
\end{prop}


\begin{proof}[Proof of Theorem~\ref{thm:main}]
We prove by induction over $k$ that for each generically exact hermitian complex
$(E,\varphi)$ of length $\leq k+p$ such that  $H_\ell(E) = 0$ for
$\ell > k$, the associated residue current satisfies
\eqref{eq:main}. Since $(E,\varphi)$ in Theorem \ref{thm:main}
has
length $N$ and $H_\ell(E)$ has pure codimension $p$, $(E,\varphi)$
has this property for $k=N$, and thus the theorem follows.

First note that the case $k = 0$ is Theorem~\ref{thm:higherrank}.
Next assume that \eqref{eq:main} holds for residue currents associated
with complexes $(E,\varphi)$ of length $k-1+p$ such that that $H_\ell(E)=
0$ for $\ell > k-1$. It is enough to prove \eqref{eq:main} locally
and since both sides in \eqref{eq:main} are pseudomeromorphic currents of bidegree $(p,p)$,
by the dimension principle,
it is enough to prove it outside a subvariety of
codimension $p+1$. Therefore we can assume that we have generically
exact hermitian complexes
$(F,\psi)$ and $(G,\eta)$ as in Lemma \ref{lma:bigDiagram}.
It follows from Theorem \ref{thm:higherrank} and
\eqref{eq:rowsHomology} that
\begin{equation*}
\frac{1}{(2\pi i)^p p!}(\tr D\psi R^F)_p=(-1)^k[H_k(E)].
\end{equation*}
Moreover, by the induction hypothesis and
\eqref{eq:rowsHomology}
\begin{equation*}
\frac{1}{(2\pi i)^p p!}(\tr D\eta R^G)_p=\sum_{\ell=0}^{k-1}(-1)^\ell[H_\ell(E)].
\end{equation*}
Now \eqref{eq:main} follows from Proposition
\ref{prop:inductionStep}.
\end{proof}

\subsection{Proof of Proposition~\ref{prop:inductionStep}}

We will compute the two terms on the right-hand side of
\eqref{eq:inductionStep} separately.


\subsubsection{Computing $(\tr D\eta R^G)_p$}\label{deta}

Let us consider the currents
\begin{equation}\label{kapten}
D\eta_{\ell+1} D\eta_{\ell+2}\cdots D\eta_{\ell+p} (R^G)^\ell_{\ell+p}
\end{equation}
that one takes the trace of in $(\tr D\eta R^G)_p$.
Note that \eqref{kapten} vanishes for $\ell\geq
k$ since then $G_{\ell+p}=0$. It remains to consider the cases $\ell=0,\ldots, k-1$. For these $\ell$,
$b_\ell$ is an isomorphism and thus it is invertible.
By the comparison formula, \eqref{cirkus1}, we have
\begin{equation*}
 (R^G)^{\ell}_{\ell+p} b_\ell   = b_{\ell+p} (R^E)^\ell_{\ell+p} + M^{\ell-1}_{\ell+p} \varphi_\ell + \eta_{\ell+p+1}M^\ell_{\ell+p+1}.
\end{equation*}
It follows that
\eqref{kapten} equals
\begin{multline*}
D\eta_{\ell+1} \cdots D\eta_{\ell+p}b_{\ell+p}
(R^E)^\ell_{\ell+p}b_\ell^{-1} + \\
D\eta_{\ell+1} \cdots D\eta_{\ell+p}M^{\ell-1}_{\ell+p} \varphi_\ell b_\ell^{-1}
+
D\eta_{\ell+1} \cdots D\eta_{\ell+p}\eta_{\ell+p+1}
M^\ell_{\ell+p+1} b_\ell^{-1}.
\end{multline*}
We rewrite the trace of the last term as
\begin{multline*}
\tr 
D\eta_{\ell+1}\cdots D\eta_{\ell+p} \eta_{\ell+p+1}
M^\ell_{\ell+p+1}b_\ell^{-1}
=
\tr 
\eta_{\ell+1} D\eta_{\ell+2}\cdots D\eta_{\ell+p+1}
M^\ell_{\ell+p+1}b_\ell^{-1}
=\\
\tr 
D\eta_{\ell+2}\cdots D\eta_{\ell+p+1}
M^\ell_{\ell+p+1}b_\ell^{-1}\eta_{\ell+1}
=
\tr 
D\eta_{\ell+2}\cdots D\eta_{\ell+p+1}
M^\ell_{\ell+p+1}\varphi_{\ell+1}b_{\ell+1}^{-1}.
\end{multline*}
Here we have used \eqref{dal} for the first equality and
\begin{equation}\label{mal}
b_{\ell}^{-1}\eta_{\ell+1}=\varphi_{\ell+1}b_{\ell+1}^{-1}
\end{equation}
for the last equality; indeed, since $\ell<k$, $b_{\ell+1}$ is invertible. For the middle equality we have used
\eqref{banal}; note that the sign is $1$ since both
$D\eta_{\ell+2}\cdots D\eta_{\ell+p+1}M^\ell_{\ell+p+1}b_\ell^{-1}$
and $\eta_{\ell+1}$ have odd total
and endomorphism degrees.
It follows that
\begin{multline*}
\sum_{\ell=0}^{k-1} (-1)^\ell \Big (\tr 
D\eta_{\ell+1}\cdots
  D\eta_{\ell+p} M^{\ell-1}_{\ell+p}\varphi_\ell b_\ell^{-1}
+ \tr 
D\eta_{\ell+1}\cdots
  D\eta_{\ell+p} \eta_{\ell+p+1} M^{\ell}_{\ell+p+1} b_\ell^{-1}
  \Big )
=\\
\tr 
D\eta_{1}\cdots
  D\eta_{p} M^{-1}_{p}\varphi_0 b_0^{-1}
+
(-1)^{k-1}\tr 
D\eta_{k}\cdots
  D\eta_{k+p-1} \eta_{k+p} M^{k-1}_{k+p} b_{k-1}^{-1}
=0,
\end{multline*}
since $\varphi_0=0$ and $\eta_{k+p}=0$.
Thus
\begin{multline}\label{furir2}
(\tr D\eta R^G)_p=\sum_{\ell=0}^{k-1} (-1)^\ell \tr D\eta_{\ell+1} \cdots D\eta_{\ell+p}
(R^G)^\ell_{\ell+p} =  \\
\sum_{\ell=0}^{k-1} (-1)^{\ell}
\tr D\eta_{\ell+1} \cdots D\eta_{\ell+p}b_{\ell+p}
(R^E)^\ell_{\ell+p}b_\ell^{-1} =
\sum_{\ell=0}^{k-1} (-1)^{\ell} \tr 
(\alpha_\ell + \beta_\ell +
\gamma_\ell) (R^E)^\ell_{\ell+p} b_{\ell}^{-1} , 
\end{multline}
where we have used Lemma ~\ref{sniken} for the last equality and
$\alpha_\ell, \beta_\ell, \gamma_\ell$ are as in the lemma.

From the definitions of $\alpha_\ell$ and $\beta_{\ell}$ it follows that
\begin{multline}\label{fromma}
\tr 
\alpha_\ell (R^E)^{\ell}_{\ell+p} b_\ell^{-1}
=
\tr 
\eta_{\ell+1}\delta_{\ell+1} (R^E)^{\ell}_{\ell+p}
b_\ell^{-1}
=
\tr 
\delta_{\ell+1} (R^E)^{\ell}_{\ell+p}
  b_\ell^{-1}\eta_{\ell+1}
= \\
\tr 
\delta_{\ell+1} (R^E)^{\ell}_{\ell+p}
  \varphi_{\ell+1}b_{\ell+1}^{-1}
=
\tr 
\delta_{\ell+1} \varphi_{\ell+p+1}(R^E)^{\ell+1}_{\ell+p+1}
  b_{\ell+1}^{-1}
=
\tr 
\beta_{\ell+1} (R^E)^{\ell+1}_{\ell+p+1}
  b_{\ell+1}^{-1}
\end{multline}
for $\ell = 0,\ldots,k-1$.
Here we have used \eqref{banal} for the second equality; indeed, note
that the sign is $1$ since $\delta_{\ell+1} (R^E)^{\ell}_{\ell+p}
b_\ell^{-1}$ and $\eta_{\ell+1}$ have odd total and endomorphism degrees.
Moreover, we have used \eqref{mal}
for the third equality and \eqref{eq:nabla2} for the fourth equality.
By \eqref{boden}, we then get that
\begin{equation}\label{fotas}
\tr 
\beta_0
  (R^E)^0_{p}b_0^{-1}
= \tr 
\delta_0 \varphi_p (R^E)^0_p b^{-1}_0
=0.
\end{equation}
Moreover note that
\begin{multline}\label{overste}
\tr 
\gamma_\ell (R^E)^\ell_{\ell+p}b_\ell^{-1}
=
\tr 
b_\ell D\varphi_{\ell+1}\cdots D\varphi_{\ell+p}
(R^E)^\ell_{\ell+p}b_\ell^{-1}
=\\
\tr 
b_\ell^{-1}b_\ell D\varphi_{\ell+1}\cdots D\varphi_{\ell+p}
(R^E)^\ell_{\ell+p}
=\tr 
D\varphi_{\ell+1}\cdots D\varphi_{\ell+p}
(R^E)^\ell_{\ell+p}, 
\end{multline}
where we have used \eqref{banal} for the second equality; indeed,
the sign is $1$ since $b_{\ell}$ is of even total and endomorphism
degree.
From \eqref{furir2}, \eqref{fromma}, \eqref{fotas}, and \eqref{overste} we conclude
\begin{equation}\label{soto}
(\tr D\eta R^G)_p
=
(-1)^{k-1}\tr 
\beta_k
  (R^E)^k_{k+p}b_k^{-1}
+\sum_{\ell=0}^{k-1} (-1)^\ell \tr 
D\varphi_{\ell+1}\cdots D\varphi_{\ell+p}
  (R^E)^\ell_{\ell+p}. 
\end{equation}

\subsubsection{Computing $(\tr D\psi R^F)_p$}\label{lilla}

Since $F_\ell=0$ for $\ell<k$, the only nonvanishing current that one takes the trace of in $(\tr D\psi R^F)_p$ is
$(-1)^k D\psi_{k+1} \cdots D\psi_{k+p} (R^F)^k_{k+p}.$
A computation yields\footnote{In the sum $D\tilde\eta_{k+2} \cdots D\tilde\eta_m$ and $D\varphi_{m+1}\cdots D\varphi_{k+p-1}$ are to be interpreted as $1$
    if $m = k+1$ and $m=k+p-1$, respectively.}
\begin{multline}\label{strunta}
D\psi_{k+1} \cdots D\psi_{k+p-1}=\\
D\left[\begin{array}{cc} -\tilde\eta_{k+2} & \varepsilon b_{k+1} \end{array}\right]
D\left[\begin{array}{cc} -\tilde\eta_{k+3} & \varepsilon b_{k+2} \\ 0 &
                                                             \varphi_{k+2} \end{array}\right]
                                                             \cdots
D\left[\begin{array}{cc} -\tilde\eta_{k+p-1} & \varepsilon b_{k+p-2} \\ 0 &
                                                             \varphi_{k+p-2} \end{array}\right]
D\left[\begin{array}{c} \varepsilon b_{k+p-1} \\
                          \varphi_{k+p-1} \end{array}\right]
=\\
\sum_{m=k+1}^{k+p-1}(-1)^{m-k-1} D\tilde\eta_{k+2}\cdots D\tilde\eta_m
D(\varepsilon b_m)
  D\varphi_{m+1}\cdots D\varphi_{k+p-1}.
\end{multline}
Recall that by \eqref{bord}
\begin{equation}\label{ren}
D(\varepsilon b_m)=D\varepsilon b_m
+(-1)^{\deg\varepsilon}\varepsilon Db_m = -\varepsilon Db_m.
\end{equation}
Moreover by \eqref{mars} we have
\begin{equation}\label{elk}
D\tilde\eta_\kappa\varepsilon=(-1)^{\deg_f (D\eta_\kappa)}
\varepsilon D\eta_\kappa=
-\varepsilon D\eta_\kappa.
\end{equation}
Using \eqref{ren} and then \eqref{elk} repeatedly for $\kappa = m,
m-1, \ldots, k+2$ we get
\begin{multline}\label{pukas}
D\tilde\eta_{k+2}\cdots D\tilde\eta_m
D(\varepsilon b_m)
  D\varphi_{m+1}\cdots D\varphi_{k+p-1}
=\\
(-1)^{m-k}
\varepsilon D\eta_{k+2}\cdots D \eta_m Db_m
  D\varphi_{m+1}\cdots D\varphi_{k+p-1}.
\end{multline}

Now
\begin{multline}\label{motvalls}
D\psi_{k+1} \cdots 
D\psi_{k+p}= \\
-\Big (\sum_{m=k+1}^{k+p-1} \varepsilon D\eta_{k+2}\cdots D \eta_m Db_m
D\varphi_{m+1}\cdots D\varphi_{k+p-1} \Big ) D\varphi_{k+p}=\\
-\sum_{m=k+1}^{k+p} \varepsilon D\eta_{k+2}\cdots D \eta_m Db_m
D\varphi_{m+1}\cdots D\varphi_{k+p}
= -\varepsilon \delta_{k+1},
\end{multline}
where $\delta_{k+1}$ is as in Lemma \ref{sniken};
here we have used $\psi_{k+p}=\varphi_{k+p}$, cf.\ \eqref{eq:bigDiagram}, \eqref{strunta}, and \eqref{pukas} for the
first equality, and $b_{k+p}=0$ for the second.
It follows that
\begin{multline*}
\tr
D\psi_{k+1} \cdots D\psi_{k+p} (R^F)^k_{k+p} 
=\tr
D\psi_{k+1} \cdots D\psi_{k+p} (R^E)^k_{k+p}a_k 
=\\
\tr
a_k D\psi_{k+1} \cdots D\psi_{k+p} (R^E)^k_{k+p}
=
  \tr
b_k^{-1} \eta_{k+1}\varepsilon^{-1} D\psi_{k+1} \cdots D\psi_{k+p}
(R^E)^k_{k+p}
=\\
  -\tr
b_k^{-1}\eta_{k+1}\varepsilon^{-1} \varepsilon \delta_{k+1}
(R^E)^k_{k+p}
= -\tr
b_k^{-1}\alpha_k (R^E)^k_{k+p},
\end{multline*}
where $\alpha_k$ is in Lemma ~\ref{sniken}.
Here we have used the comparison formula \eqref{cirkus3} and that
$a_{k+p}=\Id_{E_{k+p}}$, \eqref{banal},
the definition of $a_k$, and \eqref{motvalls} for
the first, second, third, and fourth equality, respectively.
Since $\eta_{k+p} = 0$, cf.\ \eqref{eq:bigDiagram}, by Lemma ~\ref{sniken} we get that
\begin{equation*}
    -b_k^{-1} \alpha_k = b_k^{-1} (\beta_k + \gamma_k)  = b_k^{-1} \beta_k + D\varphi_{k+1}\cdots D\varphi_{k+p}.
\end{equation*}
Thus
\begin{multline} \label{eq:DxiRK}
(\tr D\psi R^F)_p
=(-1)^k \tr
D\psi_{k+1} \cdots D\psi_{k+p} (R^F)^k_{k+p}
=\\
(-1)^k \tr
b_k^{-1} \beta_k  (R^E)^k_{k+p}
+(-1)^k \tr
D\varphi_{k+1}\cdots D\varphi_{k+p} (R^E)^k_{k+p}. 
\end{multline}

\smallskip

Finally using \eqref{banal} we conclude from \eqref{soto} and
\eqref{eq:DxiRK} that
\begin{multline*}
(\tr D\psi R^F)_p+(\tr D\eta R^G)_p=
(-1)^k \tr b_k^{-1} \beta_k  (R^E)^k_{k+p}+
(-1)^{k-1}\tr \beta_k (R^E)^k_{k+p}b_k^{-1}\\
+\sum_{\ell=0}^{k} (-1)^\ell \tr
D\varphi_{\ell+1}\cdots D\varphi_{\ell+p}
  (R^E)^\ell_{\ell+p} = (\tr D\varphi R^E)_p.
\end{multline*}

\section{The Koszul complex}\label{koszulsection}
Let $(E, \varphi)$ be the Koszul complex of a tuple $f$ of holomorphic
functions as in Section \ref{ssect:koszul}, and assume that it is
equipped with the trivial metric.
Recall that if $m=p=\codim Z(f)$, then \eqref{eq:main} just equals \eqref{chpl}.
In this section we describe
the currents in \eqref{eq:main} when $m>p$.
Then
   \begin{equation*}
    [E]=\sum (-1)^\ell [\mathcal H_\ell(E)] = 0,
    \end{equation*}
see, e.g., \cite{Roberts}*{Corollary~5.2.9~(ii)};
in particular the Koszul complex cannot be exact at all levels
$\ell>0$.

To describe the left-hand side of \eqref{eq:main}, let us recall the
construction of $R$.
Let $\sigma = \sum \bar{f_i}e_i/|f|^2$. Then $R^\ell_k$ is defined as
multiplication with
the analytic continuation to $\lambda=0$ of the form
$\dbar|f|^{2\lambda}\wedge \sigma\wedge (\dbar\sigma)^{k-\ell-1}$, see \cite{AndIdeals}.
Since $D\varphi_j$ is just contraction with
$\sum df_j \wedge e^*_j$,
a computation yields that
\begin{equation*}
\tr D\varphi_{\ell+1} \cdots D\varphi_{k} R^\ell_{k} =
\binom{m-(k-\ell)}{\ell}
\tr D\varphi_1 \cdots D\varphi_{k-\ell} R^0_{k-\ell},
\end{equation*}
cf.\ \eqref{disputation}.
Since $(E,\varphi)$ ends at level $m$, $(2\pi i)^p p!$ times the left-hand side of
\eqref{eq:main} equals
\begin{equation*}\label{snara}
        \sum_{\ell=0}^{m-p} (-1)^\ell \tr D\varphi_{\ell+1} \cdots
        D\varphi_{\ell+p} R^\ell_{\ell+p} =
\sum_{\ell=0}^{m-p} (-1)^{\ell} \binom{m-p}{\ell}  \tr D\varphi_1 \cdots D\varphi_p R^0_p = 0.
    \end{equation*}
To conclude, \eqref{eq:main} holds since both sides vanish, so we get
an explicit proof of Theorem ~\ref{thm:main} in this case.

Next, let us consider the individual terms in the left-hand side of
\eqref{eq:main} and in $[E]$. First note that since the image of
$\varphi_1$ equals $\J(f)$,  $[\mathcal
H_0(E)]$ is just the cycle of $\Ok_{Z}=\Ok_X/\J(f)$.
In \cite{ALelong}, Andersson proved that
    \begin{equation} \label{AndPL}
        \frac{1}{(2\pi i)^p p!} \tr D\varphi_1 \cdots D\varphi_p R^0_p = \sum \alpha_j [Z_j^p],
    \end{equation}
    where $Z_j^p$ are the irreducible components of $Z$
of codimension $p$ and $\alpha_j$
    is the \emph{geometric} or \emph{Hilbert-Samuel} multiplicity of $\mathcal{J}(f)$ along $Z_j^p$.
    For a complete intersection ideal, the geometric multiplicities
    coincide with the algebraic multiplicities and so \eqref{AndPL}
    generalizes \eqref{chpl}. In general, however, the multiplicities are different,
    cf.\ Example \ref{strumpa} below, and thus it is not true in
    general that the individual terms
$\tr D\varphi_{\ell+1} \cdots D\varphi_{\ell+p} R^\ell_{\ell+p}$ and
$[\mathcal H_\ell (E)]$ at level $\ell$ coincide.

\begin{ex}\label{strumpa}
If $\J(f)$ is generated by monomials and $Z(f) = \{ 0 \}$, then the
algebraic multiplicity 
equals $n!$ times the volume of
$\R^n_+\setminus \Gamma$, where the $\Gamma$ is the convex hull in
$\R^n$ of the exponents of the monomials in $\J(f)$, see, e.g.,
\cite{Roberts}*{exercise~2.8}.  If $\J(f)$ is not a
complete intersection ideal, this does not coincide with the geometric
multiplicity, which is just the number of monomials that are not in
$\J(f)$.

For example, if $f=(z_1^2,z_1 z_2, z_2^2)$ in $\C^2$, then the
algebraic multiplicity of $\J(f)$ is $4$, while the geometric multiplicity is
$3$. Thus in this case the first term in \eqref{eq:main} equals
\begin{equation*}
 \frac{1}{(2\pi i)^2 2!}\tr D\varphi_1 D\varphi_2 (R^E)^0_2 = 4[0]
\end{equation*}
whereas $[\mathcal{H}_0(E)]=3[0]$.
\end{ex}

\section{Non-pure dimensional homology}\label{ickeren}

In \cite{LW} we get a version \cite{LW}*{Theorem~1.5} of
Theorem~\ref{thm:higherrank} when $\mathcal{F} = \Ok_Z$
for a general, not necessarily pure dimensional, analytic space $Z$.
By the same arguments we get a version for general coherent sheaves
$\mathcal F$.
\begin{cor}\label{regn}
Let $\mathcal F$ be a coherent sheaf, let $(E,\varphi)$ be a hermitian
locally free resolution of $\mathcal F$,
    and let $D$ be the connection on $\End E$ induced by arbitrary $(1,0)$-connections on $E_0,\ldots,E_N$.
Moreover, let $W_k$ be the union of all irreducible components of $\supp
\mathcal F$ 
of codimension $k$.
Then
\begin{equation} \label{eqNonPure}
    \sum_k \frac{1}{(2\pi i)^k k!} \tr D\varphi_1 \cdots D\varphi_k
    {\mathbf 1_{W_k}} R^0_k = [\mathcal F].
\end{equation}
\end{cor}
Pseudomeromorphic currents allow for multiplication by characteristic
functions of varieties or, more generally, constructible sets,  see
\cite[Theorem~3.1]{AW2}, and thus
${\mathbf 1_{W_k}} R^0_k$ is a well-defined pseudomeromorphic
current.

It is natural to ask whether we also obtain a version of
Theorem~\ref{thm:main} when the homology groups do not have pure
dimension or are not of the same dimension. However, this does not seem to follow as easily.
Since $[E]$ is an alternating sum of cycles of sheaves, there are in
general components $m_i[Z_i]$ and $m_j[Z_j]$ of $[E]$ such that
$Z_i$ is a proper subvariety of $Z_j$. If we remove these ``embedded
components'' of $[E]$ we can get a formula like \eqref{eq:main}:
Let $W = \cup \supp \mathcal{H}_\ell(E)$ and let
\[
[E]_{W}=\sum (-1)^\ell [\mathcal{H}_\ell(E)]_{W},
\]
where $[\mathcal{H}_\ell(E)]_W$ is the cycle of
$\mathcal{H}_\ell(E)$ but where we
only include the irreducible components $Z_i$ that are minimal primes of $W$.
Moreover, let $W_k$ be the union of the irreducible components of $W$
of codimension $k$.
Then by the same arguments 
as in the proof of \cite{LW}*{Theorem~1.5} we get
\begin{equation} \label{trubbla}
    \sum_{k,\ell} \frac{1}{(2\pi i)^k k!}  \tr D\varphi_{\ell+1} \cdots
    D\varphi_{\ell+k} {\mathbf 1_{W_k}} R^{\ell}_{\ell+k} =
[E]_{W}.
\end{equation}
Maybe one could get a similar formula for $[E]$ by
considering characteristic functions of
%
different sets at different levels. For example if $W^\ell_k$ is the union of
the irreducible components of $\supp \mathcal{H}_\ell(E)$ of
codimension $k$,
one could hope that
\begin{equation*}
    \sum_{k,\ell} \frac{1}{(2\pi i)^k k!}
    \tr D\varphi_{\ell+1} \cdots D\varphi_{\ell+k} {\mathbf 1_{W^\ell_k}}
    R^{\ell}_{\ell+k} = [E]. 
\end{equation*}
However, this does not seem to follow as immediately from the dimension
principle as ~\eqref{trubbla}.

\end{document}